\documentclass[a4paper,11pt,centertags,reqno,psamsfonts]{amsart}

\usepackage{letltxmacro}          
\usepackage{ifthen}				
\usepackage[final]{showkeys}      
\usepackage[hmargin=1in,vmargin=1in]{geometry} 
\usepackage{setspace}				
\usepackage{array}				
\usepackage{fancybox}				
\usepackage{endnotes}             
\usepackage{booktabs}             
\usepackage{tabularx}             
\usepackage{caption}              
\usepackage{rotating}             
\usepackage{adjustbox}            
\usepackage{subfloat}             
\usepackage{graphicx}				
\usepackage{subfigure}			
\usepackage{amsmath}				
\usepackage{amsthm}				
\usepackage{amssymb}				
\usepackage{mathtools}			
\usepackage{mathrsfs}				
\usepackage{stmaryrd}				
\usepackage{yhmath}				
\usepackage{accents}				
\usepackage[all]{xy}				
\usepackage[round,comma,authoryear,sort&compress]{natbib} 

\newcommand{\N}{\mathbb{N}}

\newcommand{\R}{\mathbb{R}}

\newcommand{\function}[3]{#1:#2\rightarrow #3}

\newcommand{\C}[2]{
\ifthenelse{#1=0 \and #2=0}{\textsf{\upshape C}}
{\ifthenelse{#2=0}{\textsf{\upshape C}^{#1}}
{\textsf{\upshape C}^{#1,#2}}}
}

\renewcommand{\d}{\mathrm{d}}

\newcommand{\e}{\mathrm{e}}

\newcommand{\E}{\textsf{\upshape E}}

\renewcommand{\P}{\textsf{\upshape P}}

\newcommand{\indicator}[1]{\mathbf{1}_{#1}}

\renewcommand{\>}{\rangle}

\newcommand{\filtration}[1]{\mathfrak{#1}}
\newcommand{\sigalgebra}[1]{\mathscr{#1}}

\newcommand{\scale}{\ensuremath{\mathfrak{s}}}
\newcommand{\speed}{\ensuremath{\mathfrak{m}}}

\makeatletter
\let\oldr@@t\r@@t
\def\r@@t#1#2{%
\setbox0=\hbox{$\oldr@@t#1{#2\,}$}\dimen0=\ht0
\advance\dimen0-0.2\ht0
\setbox2=\hbox{\vrule height\ht0 depth -\dimen0}%
{\box0\lower0.4pt\box2}}
\LetLtxMacro{\oldsqrt}{\sqrt}
\renewcommand*{\sqrt}[2][\ ]{\oldsqrt[#1]{#2}}
\makeatother

\theoremstyle{plain}
\newtheorem{theorem}{Theorem}
\newtheorem{lemma}[theorem]{Lemma}
\newtheorem{proposition}[theorem]{Proposition}

\theoremstyle{definition}

\newtheorem{example}[theorem]{Example}

\theoremstyle{remark}
\newtheorem{remark}[theorem]{Remark}

\numberwithin{theorem}{section}
\numberwithin{equation}{section}
\numberwithin{figure}{section}
\numberwithin{table}{section}

\graphicspath{{Figures/}}

\bibpunct{(}{)}{,}{a}{}{,}

\allowdisplaybreaks[4]

\begin{document}
\title{Optimal Prediction of the Last-Passage Time of a Transient Diffusion}

\author{Kristoffer Glover}
\author{Hardy Hulley}

\address{Kristoffer Glover\\
Finance Discipline Group\\
University of Technology, Sydney\\
P.O. Box 123\\
Broadway, NSW 2007\\
Australia}
\email{kristoffer.glover@uts.edu.au}

\address{Hardy Hulley\\
Finance Discipline Group\\
University of Technology, Sydney\\
P.O. Box 123\\
Broadway, NSW 2007\\
Australia}
\email{hardy.hulley@uts.edu.au}

\subjclass[2010]{Primary: 60G40, 60J60, 60J65; Secondary: 34A34, 49J40, 60G44}

\keywords{Transient diffusions; last-passage times; optimal prediction; optimal stopping; free-boundary problems}

\date{\today}

\begin{abstract}
We identify the integrable stopping time $\tau_*$ with minimal $L^1$-distance to the last-passage time $\gamma_z$ to a given level $z>0$, for an arbitrary non-negative time-homogeneous transient diffusion $X$. We demonstrate that $\tau_*$ is in fact the first time that $X$ assumes a value outside a half-open interval $[0,r_*)$. The upper boundary $r_*>z$ of this interval is characterised either as the solution for a one-dimensional optimisation problem, or as part of the solution for a free-boundary problem. A number of concrete examples illustrate the result.
\end{abstract}

\maketitle

\section{Introduction}
\label{sec1}
Let $\function{a}{(0,\infty)}{(0,\infty)}$ and $\function{b}{(0,\infty)}{\R_+}$ satisfy the local integrability condition
\begin{equation}
\label{eqsec1:LocInt}
\forall x>0,\;\exists\varepsilon>0\qquad\text{such that}\qquad\int_{x-\varepsilon}^{x+\varepsilon}\frac{1+|b(y)|}{a^2(y)}\,\d y<\infty.
\end{equation}
It follows that the stochastic differential equation
\begin{equation}
\label{eqsec1:SDE}
\d X_t=b(X_t)\,\d t+a(X_t)\,\d B_t,
\end{equation}
for all $t\geq 0$, admits a unique non-negative weak solution $(X,B)=(X_t,B_t)_{t\geq 0}$ up to its (possibly infinite) explosion time, with uniqueness in the sense of probability law \citep[see e.g.][Theorem~5.5.15]{KS91}. The process $B$ is a standard Brownian motion on a complete filtered probability space $(\Omega,\sigalgebra{F},\filtration{F},\P)$, whose filtration $\filtration{F}=(\sigalgebra{F}_t)_{t\geq 0}$ satisfies the ``usual conditions'' of right continuity and completion by the null-sets of $\sigalgebra{F}$. The explosion time of $X$ is given by $\zeta\coloneqq\lim_{n\uparrow\infty}\zeta_n$, where
\begin{equation*}
\zeta_n\coloneqq\inf\biggl\{t\geq 0\,\Bigl|\,X_t\notin\biggl(\frac{1}{n},n\biggr)\biggr\},
\end{equation*}
for each $n\in\N$. We assume that $X$ is killed when it explodes.

The scale function $\scale$ and the speed measure $\speed$ of $X$ are determined by
\begin{equation}
\label{eqsec1:ScaleSpeed1}
\scale(x)\coloneqq\int^x\e^{-\int^y\frac{2b(\xi)}{a^2(\xi)}\,\d\xi}\,\d y
\qquad\text{and}\qquad
\speed(\d x)\coloneqq\frac{2\,\d x}{a^2(x)\scale'(x)},
\end{equation}
for all $x>0$. We assume that $\scale$ and $\speed$ satisfy the following conditions:
\begin{equation}
\label{eqsec1:Assump}
\scale(0+)=-\infty,\qquad\scale(\infty-)=0,\qquad\text{and}\qquad
\int_0^1\speed(\d y)<\infty.
\end{equation}
The first two conditions ensure that $X$ is transient, in the sense that
\begin{equation*}
\P_x\biggl(\lim_{t\uparrow\zeta}X_t=\infty\biggr)=1
\qquad\text{and}\qquad
\P_x\biggl(\inf_{0\leq t<\zeta}X_t>0\biggr)=1,
\end{equation*}
for all $x>0$ \citep[see e.g.][Proposition~5.5.22]{KS91}. Here $\P_x$ denotes the probability measure under which $X_0=x$. The first and third conditions in Assumption~\eqref{eqsec1:Assump} imply that the origin is either a non-attractive natural boundary, or an entrance boundary \citep[see e.g.][Table~15.6.2]{KT81}.

Fix $z>0$, and consider the last-passage time
\begin{equation*}
\label{eqsec1:LPTime}
\gamma_z\coloneqq\sup\{t\geq 0\,|\,X_t=z\}.
\end{equation*}
The transience of $X$ ensures that $\gamma_z<\infty$ $\P_x$-a.s., for all $x>0$. It also implies that $\gamma_z>0$ $\P_x$-a.s., whenever $0<x<z$. Note that $\gamma_z$ is not a stopping time, since it is impossible to determine that a sample path of $X$ has reached $z$ for the last time, without knowledge of its future evolution. It is thus natural to enquire whether $\gamma_z$ can be approximated optimally, in some sense, by a stopping time. This leads to the formulation of the following optimal stopping problem, which is the subject of our investigation:
\begin{equation}
\label{eqsec1:OptPredProb1}
V(x)\coloneqq\inf_\tau\E_x(|\gamma_z-\tau|),
\end{equation}
for all $x>0$, where the infimum is computed over all integrable stopping times.

The optimal stopping problem presented in \eqref{eqsec1:OptPredProb1} belongs to the class of so-called optimal prediction problems, the defining characteristic of which is that their payoffs are unknown at stopping times. The progenitor of optimal prediction is \citet{GPS01}, which solves the problem of stopping a Brownian motion as close as possible to its ultimate maximum over a finite time-interval, using an $L^2$-criterion. That problem was later revisited by \citet{Ped03} and \citet{Uru05}, who considered alternative criteria for optimality.

Subsequent research focussed mainly on optimally predicting the maxima and minima of arithmetic and geometric Brownian motions. To begin with, \citet{DP07} investigated the problem of stopping an arithmetic Brownian motion as close as possible to its ultimate maximum, in an $L^2$-sense, with a finite time-horizon. Thereafter, \citet{DP08} identified the stopping time with minimal $L^1$-distance to the time at which an arithmetic Brownian motion achieves its ultimate maximum, over a finite time-interval. Next, \citet{SXZ08} and \citet{DP09} obtained solutions for the problem of stopping a geometric Brownian motion in order to maximise the expected ratio of its stopped value to its ultimate maximum, over a finite time-interval. They also solved the problem of stopping a geometric Brownian motion in order to minimise the expected inverse of that ratio.

Subsequently, \citet{DJZZ10} employed PDE methods to solve the same optimal prediction problems for a geometric Brownian motion, as well as analogous problems involving its ultimate minimum before a finite time. More recently, \citet{DYZ12} also used PDE methods to obtain a numerical solution for the problem of stopping a geometric Brownian motion as close as possible to its ultimate maximum, in an $L^2$-sense, with a finite time-horizon, while \citet{DZ12} solved the problem of stopping a geometric Brownian motion in order to maximise/minimise the ratio of its stopped value to its ultimate geometric/arithmetic average, over a finite time-interval. Finally, \citet{Coh10} extended a number of the previous optimal prediction problems for arithmetic and geometric Brownian motions to an infinite time-horizon, by including exponential discounting in their objective functions.

A few studies have analysed optimal prediction problems for the maxima and minima of processes other than arithmetic and geometric Brownian motions. For example, \citet{BDP11} solved the problem of stopping a stable L\'evy process of index $\alpha\in(1,2)$ as close as possible to its ultimate maximum, in an $L^p$-sense ($1<p<\alpha$), over a finite time-interval, while \citet{ET12} investigated the problem of stopping a mean-reverting diffusion in order to minimise the expected value of a convex loss function applied to the difference between its ultimate maximum and its stopped value, over a finite time-interval. Finally, \citet{GHP13} identified the stopping time with minimal $L^1$-distance to the time at which a positive transient diffusion achieves its ultimate minimum, over an infinite time-horizon. 

The problem of predicting the maximum or minimum of a stochastic process has obvious financial applications. This was first pointed out by \citet{Shi02}, who argued that optimal prediction provides a theoretical foundation for technical analysis. For example, choosing the optimal time to buy or sell a financial security before a given date, when its price follows a geometric Brownian motion, motivated the studies by \citet{DJZZ10}, \citet{DYZ12}, \citet{DZ12}, \citet{DP09}, and \citet{SXZ08}. An infinite time-horizon version of the above-mentioned problem was also considered by \citet{Coh10}. In addition, the latter study formulated the problem of measuring the regret associated with holding a toxic liability, whose value follows an arithmetic Brownian motion, in terms of optimal prediction. Finally, as a financial application of their results, \citet{GHP13} considered the infinite time-horizon problem of determining the optimal time to sell a risky security whose discounted price is a strict local martingale under a risk-neutral probability measure. Such a security may be interpreted as an example of an asset bubble, and the authors focussed on a version of the constant elasticity of variance model as a specific instance.

So far, only \citet{DPS08} have studied the problem of predicting the last-passage time of a stochastic process. Specifically, they identified the stopping time that minimises the $L^1$-distance to the last time an arithmetic Brownian motion reaches the origin, over a finite time-interval. That problem is the closest relative in the existing literature on optimal prediction to the problem studied here.

The remainder of the article is structured as follows. First, Section~\ref{sec2} reformulates the optimal prediction problem \eqref{eqsec1:OptPredProb1} under consideration, so that it is well-defined for all time-homogeneous diffusions satisfying Assumption~\eqref{eqsec1:Assump}. This more general problem is then solved in Section~\ref{sec3}, subject to the additional constraint that we restrict our attention to rules where $X$ is stopped as soon as its value exceeds some constant threshold. In Section~\ref{sec4} we demonstrate that the optimal threshold rule described above is in fact the optimal stopping policy for the general problem. Finally, Section~\ref{sec5} presents a number of illustrative examples.
\section{Reformulating the Problem}
\label{sec2}
Unfortunately, Problem~\eqref{eqsec1:OptPredProb1} is not well-defined for all diffusions satisfying the three conditions in Assumption~\eqref{eqsec1:Assump}. We shall substantiate this claim in detail, by appealing to the example of a Bessel process of dimension three, for which the above-mentioned conditions can be verified by inspecting Example~\ref{exsec5:Bes}.

To begin with, let
\begin{equation*}
\tau_y\coloneqq\inf\{t\geq 0\,|\,X_t=y\}
\end{equation*}
denote the first-passage time of $X$ to $y>0$. The following lemma computes the expected values of such stopping times, for the case when $X$ is a Bessel process of dimension three, started at the origin:

\begin{lemma}
\label{lemsec2:MeanFPTBes3D0}
Let $X$ be a Bessel process of dimension three. Then $\E_0(\tau_y)=y^2/3$, for all $y>0$.
\end{lemma}
\begin{proof}
According to \citet{HM13},
\begin{equation*}
\P_0(\tau_y>t)=\frac{\sqrt{2}}{\Gamma(3/2)}\sum_{k=1}^\infty\frac{1}{\sqrt{j_{1/2,k}}J_{3/2}(j_{1/2,k})}
\e^{-\frac{j_{1/2,k}^2}{2y^2}t},
\end{equation*}
for all $t\geq 0$ and all $x>0$, where $\Gamma$ denotes the gamma function, $J_\nu$ denotes the Bessel function of the first kind of order $\nu$, and $0<j_{\nu,1}<j_{\nu,2}<\ldots$ are the positive roots of $J_\nu$. Since
\begin{equation*}
J_{1/2}(x)=\sqrt{\frac{2}{\pi x}}\sin x,
\end{equation*}
for all $x\in\R$ \citep[see e.g.][Section~4.6]{AAR99}, it follows that $j_{1/2,k}=k\pi$, for each $k\in\N$. Furthermore, since
\begin{equation*}
J_{3/2}(x)=\sqrt{\frac{2}{\pi x}}\frac{\sin x-x\cos x}{x},
\end{equation*}
for all $x\in\R$ \citep[see e.g.][Section~4.6]{AAR99}, we obtain
\begin{equation*}
J_{3/2}(j_{1/2,k})=(-1)^{k-1}\sqrt{\frac{2}{k\pi^2}},
\end{equation*}
for each $k\in\N$. Combining these observations with $\Gamma(3/2)=\sqrt{\pi}/2$ yields
\begin{equation*}
\P_0(\tau_y>t)=2\sum_{k=1}^\infty(-1)^{k-1}\e^{-\frac{(k\pi)^2}{2y^2}t},
\end{equation*}
for all $t\geq 0$ and all $y>0$. Finally, we get
\begin{equation*}
\E_0(\tau_y)=\int_0^\infty\P_0(\tau_y>t)\,\d t
=\frac{4y^2}{\pi^2}\sum_{k=1}^\infty\frac{(-1)^{k-1}}{k^2}=\frac{y^2}{3},
\end{equation*}
for all $y>0$, with the help of \citet{Fin03}, Section~1.6, for the final equality.
\end{proof}

\begin{remark}
The appearance of Dirichlet's eta function
\begin{equation*}
\eta(x)\coloneqq\sum_{k=1}^\infty\frac{(-1)^{k-1}}{k^x},
\end{equation*}
for all $x\in\R$, in the previous proof is interesting. Specifically, the expected value of $\tau_y$ admits the following representation, when $X$ is a Bessel process of dimension three, started at the origin: 
\begin{equation*}
\E_0(\tau_y)=\frac{4y^2}{\pi^2}\eta(2),
\end{equation*}
for all $y>0$.
\end{remark}

Next, we demonstrate that $\gamma_z$ is not integrable when $X$ is a Bessel process of dimension three, started at the origin:

\begin{lemma}
\label{emsec2:MeanLPTBes3D0}
Let $X$ be a Bessel process of dimension three. Then $\E_0(\gamma_z)=\infty$.
\end{lemma}
\begin{proof}
From \citet{Get79} we get
\begin{align*}
\E_0\Bigl(\indicator{\{\gamma_z\leq T\}}\gamma_z\Bigr)
&=\int_0^Tt\P_0(\gamma_z\in\d t)
=\int_0^Tt\frac{z}{\sqrt{2\pi}}\frac{1}{t^{3/2}}\e^{-\frac{z^2}{2t}}\,\d t\\
&=\frac{2z}{\sqrt{2\pi}}\biggl(\sqrt{T}\e^{-\frac{z^2}{2T}}-\sqrt{2\pi}z\biggl(1-\Phi\biggl(\frac{z}{\sqrt{T}}\biggr)\biggr)\biggr),
\end{align*}
for all $T>0$, where $\Phi$ denotes the CDF of a standard normal random variable. It then follows that
\begin{equation*}
\E_0(\gamma_z)=\lim_{T\uparrow\infty}\E_0\Bigl(\indicator{\{\gamma_z\leq T\}}\gamma_z\Bigr)=\infty,
\end{equation*}
by an application of the monotone convergence theorem.
\end{proof}

Using Lemma~\ref{lemsec2:MeanFPTBes3D0} and the strong Markov property, we are able to extend Lemma~\ref{emsec2:MeanLPTBes3D0} to the case of a Bessel process of dimension three, with an arbitrary initial value:

\begin{proposition}
\label{propsec2:MeanLPTBes3Dx}
Let $X$ be a Bessel process of dimension three. Then $\E_x(\gamma_z)=\infty$, for all $x>0$.
\end{proposition}
\begin{proof}
Let $x\in(0,z]$, and note that $\tau_x\leq\gamma_z<\infty$ $P_0$-a.s. For $P_0$-a.a. $\omega\in\Omega$, we then obtain
\begin{align*}
\gamma_z(\omega)&=\sup\{t\geq 0\,|\,X_t(\omega)=z\}
=\tau_x(\omega)+\sup\{t\geq 0\,|\,X_{\tau_x(\omega)+t}(\omega)=z\}\\
&=\tau_x(\omega)+\sup\{t\geq 0\,|\,X_t(\vartheta_{\tau_x}\omega)=z\}
=\tau_x(\omega)+\gamma_z(\vartheta_{\tau_x}\omega),
\end{align*}
whence
\begin{equation}
\label{eqpropsec2:MeanLPTBes3Dx1}
\gamma_z=\tau_x+\gamma_z\circ\vartheta_{\tau_x}\qquad\text{$\P_0$-a.s.},
\end{equation}
where $\vartheta_{\tau_x}$ denotes the shift operator associated with the stopping time $\tau_x$ \citep[see e.g.][Section~II.4.3.4]{PS06}. Consequently,
\begin{equation*}
\E_x(\gamma_z)=\E_0(\E_x(\gamma_z))=\E_0(\E_0(\gamma_z\circ\vartheta_{\tau_x}\,|\,\sigalgebra{F}_{\tau_x}))=\E_0(\gamma_z\circ\vartheta_{\tau_x})=\E_0(\gamma_z)-\E_0(\tau_x)=\infty,
\end{equation*}
by virtue of Lemmas~\ref{lemsec2:MeanFPTBes3D0} and \ref{emsec2:MeanLPTBes3D0}, and the strong Markov property of $X$. Next, let $x>z$. A similar argument to the one above establishes the identity
\begin{equation}
\label{eqpropsec2:MeanLPTBes3Dx2}
\indicator{\{\tau_z<\infty\}}\gamma_z=\indicator{\{\tau_z<\infty\}}(\tau_z+\gamma_z\circ\vartheta_{\tau_z})\qquad\text{$\P_x$-a.s.}
\end{equation}
It then follows that
\begin{align*}
\E_x(\gamma_z)&=\E_x\Bigl(\indicator{\{\gamma_z>0\}}\gamma_z\Bigr)
=\E_x\Bigl(\indicator{\{\tau_z<\infty\}}\gamma_z\Bigr)
=\E_x\Bigl(\indicator{\{\tau_z<\infty\}}\tau_z\Bigr)+\E_x\Bigl(\indicator{\{\tau_z<\infty\}}\gamma_z\circ\vartheta_{\tau_z}\Bigr)\\
&=\E_x\Bigl(\indicator{\{\tau_z<\infty\}}\tau_z\Bigr)+\E_x\Bigl(\indicator{\{\tau_z<\infty\}}\E_x(\gamma_z\circ\vartheta_{\tau_z}\,|\,\sigalgebra{F}_{\tau_z})\Bigr)\\
&=\E_x\Bigl(\indicator{\{\tau_z<\infty\}}\tau_z\Bigr)+\E_x\Bigl(\indicator{\{\tau_z<\infty\}}\E_x(\gamma_z)\Bigr)
=\E_x\Bigl(\indicator{\{\tau_z<\infty\}}\tau_z\Bigr)+\P_x(\tau_z<\infty)\E_z(\gamma_z)
=\infty,
\end{align*}
since the recurrence of $X$ ensures that $\P_x(\tau_z<\infty)>0$, while $\E_z(\gamma_z)=\infty$ was established earlier.
\end{proof}

\begin{remark}
The following identity for shift operators is well-known:
\begin{equation*}
\tau=\sigma+\tau\circ\vartheta_\sigma,
\end{equation*}
where $\tau$ is the first-entry time into a set and $\sigma$ is an arbitrary stopping time, such that $\sigma\leq\tau$ \citep[see e.g.][Equations~4.1.25 and 7.0.7]{PS06}. Although \eqref{eqpropsec2:MeanLPTBes3Dx1} and \eqref{eqpropsec2:MeanLPTBes3Dx2} are strongly reminiscent of this identity, the fact that $\gamma_z$ is not a stopping time prevents us from applying it directly.
\end{remark}

Now, suppose that $X$ is a Bessel process of dimension three, and let $\tau$ be an arbitrary integrable stopping time. Then Proposition~\ref{propsec2:MeanLPTBes3Dx} implies that
\begin{equation*}
\E_x(|\gamma_z-\tau|)\geq\E_x(\gamma_z)-\E_x(\tau)=\infty,
\end{equation*}
for all $x>0$, whence $V(x)=\infty$. In other words, Problem~\eqref{eqsec1:OptPredProb1} is not a well-defined optimal stopping problem.

We can remedy the issue highlighted above by formulating an optimal stopping problem that is equivalent to Problem~\eqref{eqsec1:OptPredProb1} whenever $\gamma_z$ is integrable, but which remains meaningful even when $\gamma_z$ is not integrable. To do so, we first observe that
\begin{equation*}
\begin{split}
|\gamma_z-\tau|&=(\gamma_z-\tau)^++(\tau-\gamma_z)^+
=\gamma_z-\gamma_z\wedge\tau+(\tau-\gamma_z)^+\\
&=\gamma_z-\int_0^\tau\indicator{\{\gamma_z>t\}}\,\d t+\int_0^\tau\indicator{\{\gamma_z\leq t\}}\,\d t
=\gamma_z+\int_0^\tau\bigl(2\indicator{\{\gamma_z\leq t\}}-1\bigr)\,\d t,
\end{split}
\end{equation*}
for any stopping time $\tau$. It is thus natural to replace Problem~\eqref{eqsec1:OptPredProb1} with the following optimal stopping problem:
\begin{equation}
\label{eqsec2:OptPredProb2}
\widehat{V}(x)\coloneqq\inf_\tau\E_x(|\gamma_z-\tau|-\gamma_z)
=\inf_\tau\E_x\biggl(\int_0^\tau\bigl(2\indicator{\{\gamma_z\leq t\}}-1\bigr)\,\d t\biggr),
\end{equation}
for all $x>0$, where the infimum is computed over all integrable stopping times $\tau$. Note that the optimal stopping time $\tau_*$ for Problem~\eqref{eqsec2:OptPredProb2} is also the solution for Problem~\eqref{eqsec1:OptPredProb1} when $\gamma_z$ is integrable. On the other hand, even when $\gamma_z$ is not integrable, we have
\begin{equation*}
\widehat{V}(x)\leq\E_x\biggl(\int_0^\tau\bigl(2\indicator{\{\gamma_z\leq t\}}-1\bigr)\,\d t\biggr)
\leq\E_x(\tau)<\infty,
\end{equation*}
where $\tau$ is an arbitrary integrable stopping time.

Rather than attacking Problem~\eqref{eqsec2:OptPredProb2} directly, we shall focus instead on the following more tractable reformulation of that problem:
\begin{proposition}
\label{propsec2:OptPredProb3}
The following optimal stopping problem is equivalent to Problem~\eqref{eqsec2:OptPredProb2}:
\begin{equation}
\label{eqpropsec2:OptPredProb3_1}
\widehat{V}(x)\coloneqq\inf_\tau\E_x\biggl(\int_0^\tau c(X_t)\,\d t\biggr),
\end{equation}
for all $x>0$, where
\begin{equation}
\label{eqpropsec2:OptPredProb3_2}
c(x)\coloneqq 1-2\biggl(\frac{\scale(x)}{\scale(z)}\wedge 1\biggr),
\end{equation}
and the infimum is computed over all integrable stopping times $\tau$.
\end{proposition}
\begin{proof}
To begin with, suppose $x>z$, in which case the monotone convergence theorem gives
\begin{equation*}
\P_x\biggl(\inf_{t\geq 0}X_t\leq z\biggr)=\P_x(\tau_z<\infty)
=\lim_{y\uparrow\infty}\P_x(\tau_z<\tau_y)
=\lim_{y\uparrow\infty}\frac{\scale(y)-\scale(x)}{\scale(y)-\scale(z)}
=\frac{\scale(x)}{\scale(z)},
\end{equation*}
by virtue of \citet{KT81}, Equation~15.3.10, and the second condition in Assumption~\eqref{eqsec1:Assump}. Consequently,
\begin{equation*}
\P_x\biggl(\inf_{t\geq 0}X_t\leq z\biggr)=\frac{\scale(x)}{\scale(z)}\wedge 1,
\end{equation*}
for all $x>0$. The Markov property of $X$ then yields
\begin{equation*}
\P_x(\gamma_z>t\,|\,\sigalgebra{F}_t)
=\P_x(\gamma_z\circ\vartheta_t>0\,|\,\sigalgebra{F}_t)
=\P_{X_t}(\gamma_z>0)
=\P_{X_t}\biggl(\inf_{s\geq 0}X_s\leq z\biggr)
=\frac{\scale(X_t)}{\scale(z)}\wedge 1,
\end{equation*}
for all $t\geq 0$ and all $x>0$. From this it follows that
\begin{equation*}
\begin{split}
\E_x\biggl(&\int_0^\tau\bigl(2\indicator{\{\gamma_z\leq t\}}-1\bigr)\,\d t\biggr)
=\E_x\biggl(\int_0^\infty\indicator{\{\tau>t\}}\bigl(2\indicator{\{\gamma_z\leq t\}}-1\bigr)\,\d t\biggr)\\
&=\int_0^\infty\E_x\Bigl(\indicator{\{\tau>t\}}\E_x\Bigl(\bigl(2\indicator{\{\gamma_z\leq t\}}-1\bigr)\,\bigl|\,\sigalgebra{F}_t\Bigr)\Bigr)\,\d t
=\int_0^\infty\E_x\Bigl(\indicator{\{\tau>t\}}\bigl(2\P_x(\gamma_z\leq t\,|\,\sigalgebra{F}_t)-1\bigr)\Bigr)\,\d t\\
&=\E_x\biggl(\int_0^\tau\Bigl(1-2\P_x(\gamma_z>t\,|\,\sigalgebra{F}_t)\Bigr)\,\d t\biggr)
=\E_x\biggl(\int_0^\tau\biggl(1-2\biggl(\frac{\scale(X_t)}{\scale(z)}\wedge 1\biggr)\biggr)\,\d t\biggr),
\end{split}
\end{equation*}
for all $x>0$ and any stopping time $\tau$.
\end{proof}

\begin{figure}[h]
\centering
\includegraphics[scale=1]{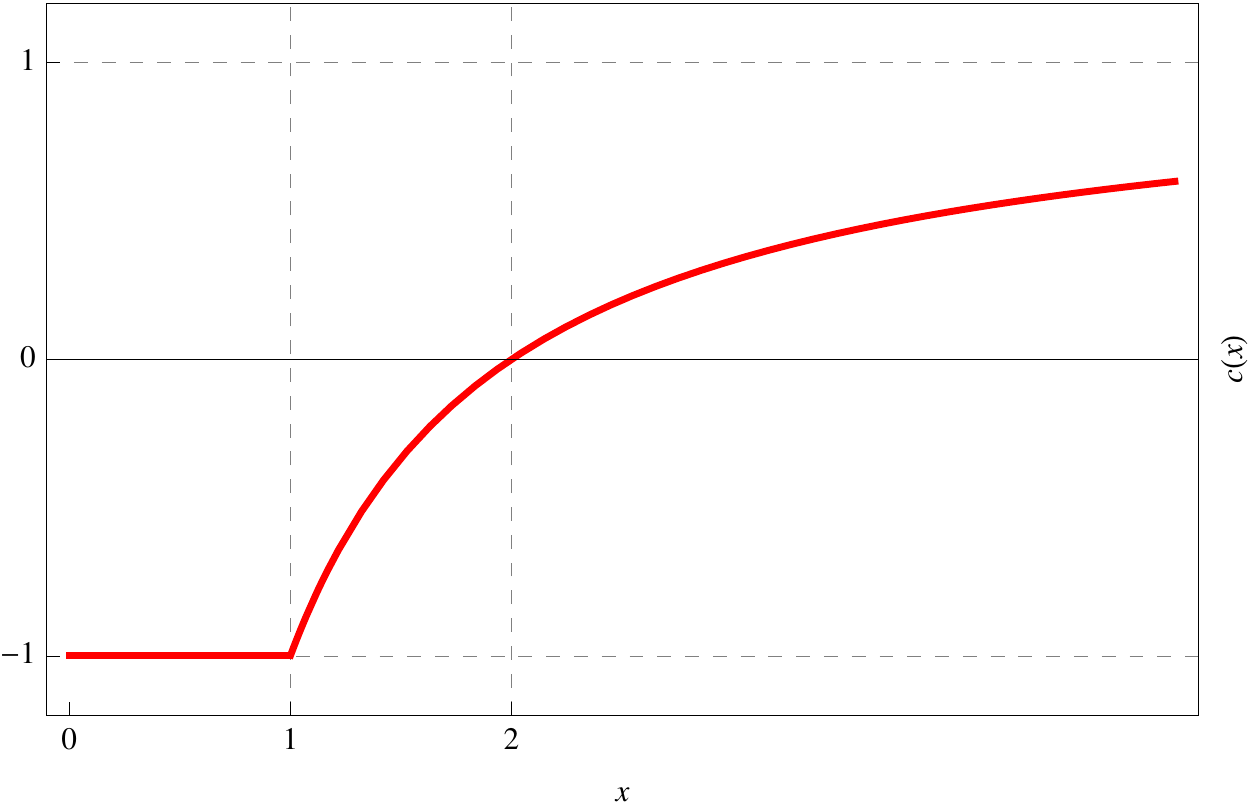}
\caption{The function $c$, defined by \eqref{eqpropsec2:OptPredProb3_2}, for a Bessel process of dimension three, with $z=1$.}
\label{figsec2:Integrand}
\end{figure}

Figure~\ref{figsec2:Integrand} plots the function $c$, given by \eqref{eqpropsec2:OptPredProb3_2}, for a Bessel process of dimension three. In that case,
\begin{equation*}
c(x)=1-2\biggl(\frac{z}{x}\wedge 1\biggr),
\end{equation*}
for all $x>0$, since $\scale(x)=-1/x$ (see Example~\ref{exsec5:Bes}). The following salient features of $c$ are illustrated by Figure~\ref{figsec2:Integrand}: First, it is monotonically increasing, with $c(x)=-1$, for all $x\in(0,z]$, and $\lim_{x\uparrow\infty}c(x)=1$. Second, it is bounded, with $|c(x)|\leq 1$, for all $x>0$. Finally, due to its monotonicity, it possesses a unique root at $\scale^{-1}\bigl(\scale(z)/2\bigr)>z$. In the case of the Bessel process of dimension three, this root is simply $2z$.
\section{Solving a Restricted Version of the Problem}
\label{sec3}
In this section we formulate and solve a restricted version of Problem~\eqref{eqpropsec2:OptPredProb3_1}, where the infimum is computed over stopping times of the form
\begin{equation*}
\sigma_r\coloneqq\inf\{t\geq 0\,|\,X_t\geq r\},
\end{equation*}
where $r>0$. To do so, we must first establish that such stopping times are in fact admissible, which means to say that they are integrable. This is achieved by setting $f(x)\coloneqq 1$, for all $x>0$, in the statement of the following lemma:

\begin{lemma}
\label{lemsec3:ExpIntRep}
Let $\function{f}{(0,\infty)}{\R}$ be a continuous function satisfying $|f(0+)|<\infty$. Then
\begin{equation*}
\E_x\biggl(\int_0^{\sigma_r}f(X_t)\,\d t\biggr)=
\begin{dcases}
\int_0^x\bigl(\scale(r)-\scale(x)\bigr)f(y)\,\speed(\d y)+\int_x^r\bigl(\scale(r)-\scale(y)\bigr)f(y)\,\speed(\d y)&\text{if $x\leq r$};\\
0&\text{if $x\geq r$},
\end{dcases}
\end{equation*}
for all $x>0$ and all $r>0$. In particular,
\begin{equation*}
\E_x\biggl(\int_0^{\sigma_r}|f(X_t)|\,\d t\biggr)<\infty,
\end{equation*}
for all $x>0$ and all $r>0$.
\end{lemma}
\begin{proof}
If $x\geq r>0$ then $\sigma_r=0$ $\P_x$-a.s., and the result follows immediately. Suppose, therefore, that $0<x<r$, and set $K\coloneqq\sup_{y\in(0,r]}|f(y)|<\infty$. Then, for each $n\in\N$, with $1/n<x<r<n$, we have
\begin{equation*}
0\leq\indicator{\{y>\frac{1}{n}\}}\frac{\bigl(\scale(r)-\scale(x)\bigr)\bigl(\scale(y)-\scale(1/n)\bigr)}{\scale(r)-\scale(1/n)}|f(y)|
\leq K\bigl(\scale(r)-\scale(x)\bigr),
\end{equation*}
for all $y\in(0,x]$, by virtue of the fact that the scale function is monotonically increasing, while
\begin{equation*}
0\leq\frac{\bigl(\scale(x)-\scale(1/n)\bigr)\bigl(\scale(r)-\scale(y)\bigr)}{\scale(r)-\scale(1/n)}|f(y)|
\leq K\bigl(\scale(r)-\scale(y)\bigr),
\end{equation*}
for all $y\in[x,r]$, for the same reason. From \citet{KT81}, Remark~15.3.3, we then obtain
\begin{equation}
\label{eqlemsec3:ExpIntRep}
\begin{split}
\E_x\biggl(\int_0^{\sigma_r\wedge\zeta_n}|f(X_t)|\,\d t\biggr)
&=\int_{1/n}^x\frac{\bigl(\scale(r)-\scale(x)\bigr)\bigl(\scale(y)-\scale(1/n)\bigr)}{\scale(r)-\scale(1/n)}|f(y)|\,\speed(\d y)\\
&\qquad\qquad\qquad+\int_x^r\frac{\bigl(\scale(x)-\scale(1/n)\bigr)\bigl(\scale(r)-\scale(y)\bigr)}{\scale(r)-\scale(1/n)}|f(y)|\,\speed(\d y)\\
&\leq\int_0^xK\bigl(\scale(r)-\scale(x)\bigr)\,\speed(\d y)
+\int_x^rK\bigl(\scale(r)-\scale(y)\bigr)\,\speed(\d y)\\
&\leq K\bigl(\scale(r)-\scale(x)\bigr)\int_0^r\speed(\d y)<\infty,
\end{split}
\end{equation}
for all $x>0$ and each $n\in\N$, such that $1/n<x<r<n$, by the monotonicity of the scale function and the third condition in Assumption~\eqref{eqsec1:Assump}. Next, successive applications of the dominated convergence theorem, combined with \citet{KT81}, Remark~15.3.3, yield
\begin{equation*}
\begin{split}
\E_x\biggl(\int_0^{\sigma_r}f(X_t)\,\d t\biggr)
&=\lim_{n\uparrow\infty}\E_x\biggl(\int_0^{\sigma_r\wedge\zeta_n}f(X_t)\,\d t\biggr)\\
&=\lim_{n\uparrow\infty}\biggl(\int_{1/n}^x\frac{\bigl(\scale(r)-\scale(x)\bigr)\bigl(\scale(y)-\scale(1/n)\bigr)}{\scale(r)-\scale(1/n)}f(y)\,\speed(\d y)\\
&\qquad\qquad\qquad+\int_x^r\frac{\bigl(\scale(x)-\scale(1/n)\bigr)\bigl(\scale(r)-\scale(y)\bigr)}{\scale(r)-\scale(1/n)}f(y)\,\speed(\d y)\biggr)\\
&=\int_0^x\lim_{n\uparrow\infty}\indicator{\left\{y>\frac{1}{n}\right\}}\frac{\bigl(\scale(r)-\scale(x)\bigr)\bigl(\scale(y)-\scale(1/n)\bigr)}{\scale(r)-\scale(1/n)}f(y)\,\speed(\d y)\\
&\qquad\qquad\qquad+\int_x^r\lim_{n\uparrow\infty}\frac{\bigl(\scale(x)-\scale(1/n)\bigr)\bigl(\scale(r)-\scale(y)\bigr)}{\scale(r)-\scale(1/n)}f(y)\,\speed(\d y)\\
&=\int_0^x\bigl(\scale(r)-\scale(x)\bigr)f(y)\,\speed(\d y)+\int_x^r\bigl(\scale(r)-\scale(y)\bigr)f(y)\,\speed(\d y),
\end{split}
\end{equation*}
for all $x>0$, due to the monotonicity of the scale function and the first condition in Assumption~\eqref{eqsec1:Assump}. Finally, note that combining Fatou's lemma with \eqref{eqlemsec3:ExpIntRep}, and the third condition in Assumption~\eqref{eqsec1:Assump}, gives
\begin{equation*}
\E_x\biggl(\int_0^{\sigma_r}|f(X_t)|\,\d t\biggr)
\leq\liminf_{n\uparrow\infty}\E_x\biggl(\int_0^{\sigma_r\wedge\zeta_n}|f(X_t)|\,\d t\biggr)
\leq K\bigl(\scale(r)-\scale(x)\bigr)\int_0^r\speed(\d y)<\infty,
\end{equation*}
for all $x>0$.
\end{proof}

\begin{remark}
Suppose, once again, that $X$ is a Bessel process of dimension three. Its scale function and speed measure are given by
\begin{equation*}
\scale(x)\coloneqq-\frac{1}{x}\qquad\text{and}\qquad\speed(\d x)\coloneqq 2x^2\,\d x,
\end{equation*}
for all $x>0$ (see Example~\ref{exsec5:Bes}). According to Lemma~\ref{lemsec3:ExpIntRep},
\begin{equation*}
\E_x(\sigma_r)
=\int_0^x\biggl(-\frac{1}{r}+\frac{1}{x}\biggr)2y^2\,\d y+\int_x^r\biggl(-\frac{1}{r}+\frac{1}{y}\biggr)2y^2\,\d y
=\frac{1}{3}(r^2-x^2),
\end{equation*}
for all $x>0$ and $r>0$, with $x\leq r$. Consequently, $\lim_{x\downarrow 0}\E_x(\sigma_r)=r^2/3$, for all $r>0$, which is consistent with Lemma~\ref{lemsec2:MeanFPTBes3D0}.
\end{remark}

We now focus on the optimal stopping problem
\begin{equation}
\label{eqsec3:OptPredProb4}
\widetilde{V}(x)\coloneqq\inf_{r>0}\E_x\biggl(\int_0^{\sigma_r}c(X_t)\,\d t\biggr),
\end{equation}
for all $x>0$. Lemma~\ref{lemsec3:ExpIntRep} ensures that the admissible stopping times for Problem~\eqref{eqsec3:OptPredProb4} are also admissible for Problem~\eqref{eqpropsec2:OptPredProb3_1}, from which it follows that $\widetilde{V}(x)\geq\widehat{V}(x)$, for all $x>0$. In the light of Lemma~\ref{lemsec3:ExpIntRep}, we also observe that
\begin{equation*}
\frac{\partial}{\partial r}\E_x\biggl(\int_0^{\sigma_r}c(X_t)\,\d t\biggr)
=\scale'(r)\int_0^rc(y)\,\speed(\d y),
\end{equation*}
for all $x>0$ and all $r>0$. This suggests the following solution for Problem~\eqref{eqsec3:OptPredProb4}:
\begin{equation}
\label{eqsec3:ValFunc1}
\widetilde{V}(x)=
\begin{dcases}
-\int_0^x\scale(x)c(y)\,\speed(\d y)-\int_x^{r_*}\scale(y)c(y)\,\speed(\d y),&\text{if $x\leq r_*$};\\
0&\text{if $x\geq r_*$},
\end{dcases}
\end{equation}
where the optimal stopping boundary $r_*>0$ is implicitly determined by the first-order equation
\begin{equation}
\label{eqsec3:OptStopBnd}
\int_0^{r_*}c(y)\,\speed(\d y)=0.
\end{equation}
In fact, since $c(x)\leq 0$, for all $x\in\bigl(0,\scale^{-1}\bigl(\scale(z)/2\bigr)\bigr]$, the inequalities $r_*>\scale^{-1}\bigl(\scale(z)/2\bigr)>z$ follow from \eqref{eqsec3:OptStopBnd} (see Figure~\ref{figsec2:Integrand}).
\section{Solving the General Version of the Problem}
\label{sec4}
This section establishes the correspondence between the function $\widetilde{V}$, given by \eqref{eqsec3:ValFunc1}, and the value function $\widehat{V}$ for the general optimal stopping problem \eqref{eqpropsec2:OptPredProb3_1}. The next result takes the first step in that direction.

\begin{proposition}
\label{propsec4:FreeBoundProb}
If the origin is a natural boundary, then the function $\widetilde{V}$, determined by \eqref{eqsec3:ValFunc1}, solves the following boundary-value problem:
\begin{subequations}
\label{eqpropsec4:FreeBoundProb1}
\begin{align}
\frac{1}{2}a^2(x)\widetilde{V}''(x)+b(x)\widetilde{V}'(x)&=-c(x);\label{eqpropsec4:FreeBoundProb1a}\\
\widetilde{V}(r_*-)&=0;\label{eqpropsec4:FreeBoundProb1b}\\
\widetilde{V}'(r_*-)&=0;\label{eqpropsec4:FreeBoundProb1c}\\
\widetilde{V}(0+)&=-\infty,\label{eqpropsec4:FreeBoundProb1d}
\end{align}
\end{subequations}
for all $x\in(0,r_*)$. Alternatively, if the origin is an entrance boundary, then the lower boundary condition \eqref{eqpropsec4:FreeBoundProb1d} is replaced with $\widetilde{V}(0+)>-\infty$.
\end{proposition}
\begin{proof}
Fix $x\in(0,r_*)$. Differentiating \eqref{eqsec3:ValFunc1} yields
\begin{equation}
\label{eqpropsec4:FreeBoundProb2}
\widetilde{V}'(x)=-\scale'(x)\int_0^xc(y)\,\speed(\d y)\qquad\text{and}\qquad
\widetilde{V}''(x)=-\scale''(x)\int_0^xc(y)\,\speed(\d y)-\frac{2}{a^2(x)}c(x).
\end{equation}
Consequently,
\begin{equation*}
\frac{1}{2}a^2(x)\widetilde{V}''(x)+b(x)\widetilde{V}'(x)
=-\biggl(\frac{1}{2}a^2(x)\scale''(x)+b(x)\scale'(x)\biggr)\int_0^xc(y)\,\speed(\d y)-c(x)
=-c(x),
\end{equation*}
since the scale function satisfies the ODE
\begin{equation*}
\frac{1}{2}a^2(x)\scale''(x)+b(x)\scale'(x)=0.
\end{equation*}
Next, $\widetilde{V}(r_*-)=0$, by inspection of \eqref{eqsec3:ValFunc1}, while $\widetilde{V}'(r_*-)=0$, by virtue of \eqref{eqsec3:OptStopBnd}. Finally, it follows from \eqref{eqpropsec2:OptPredProb3_2} and \eqref{eqsec3:ValFunc1} that
\begin{equation*}
\widetilde{V}(0+)=\lim_{x\downarrow 0}\int_0^x\scale(y)\,\speed(\d y).
\end{equation*}
Consequently, $\widetilde{V}(0+)=-\infty$ if the origin is a natural boundary, while $\widetilde{V}(0+)>-\infty$ if the origin is an entrance boundary, according to \citet{KT81}, Table~15.6.2.
\end{proof}

There is, of course, a well-established correspondence between optimal stopping problems and free-boundary problems. Seen in that light, \eqref{eqpropsec4:FreeBoundProb1} will turn out to be the free-boundary problem associated with the optimal stopping problem \eqref{eqpropsec2:OptPredProb3_1}. Next, we show that the function $\widetilde{V}$, given by \eqref{eqsec3:ValFunc1}, is non-positive:

\begin{lemma}
\label{lemsec4:ValFuncNegative}
The function $\widetilde{V}$, determined by \eqref{eqsec3:ValFunc1}, satisfies $\widetilde{V}(x)\leq 0$, for all $x>0$.
\end{lemma}
\begin{proof}
Fix $x>0$, and consider the following three cases:
\vspace{2mm}\newline\noindent
(i)~Suppose $x\geq r_*$. Then $\widetilde{V}(x)=0$ follows trivially from \eqref{eqsec3:ValFunc1}.
\vspace{2mm}\newline\noindent
(ii)~Suppose $\scale^{-1}\bigl(\scale(z)/2\bigr)\leq x\leq r_*$. Then $c(y)\geq 0$, for all $y\in[x,r_*]$, whence
\begin{equation*}
\widetilde{V}(x)=-\int_0^x\scale(x)c(y)\,\speed(\d y)-\int_x^{r_*}\scale(y)c(y)\,\speed(\d y)
\leq-\scale(x)\int_0^{r_*}c(y)\,\speed(\d y)=0,
\end{equation*}
by virtue of \eqref{eqsec3:ValFunc1}, the monotonicity of the scale function, and \eqref{eqsec3:OptStopBnd}.
\vspace{2mm}\newline\noindent
(iii)~Suppose $0<x\leq\scale^{-1}\bigl(\scale(z)/2\bigr)$. Then $c(y)\leq 0$, for all $y\in(0,x]$, whence
\begin{equation*}
\widetilde{V}'(x)=-\scale'(x)\int_0^xc(y)\,\speed(\d y)\geq 0,
\end{equation*}
due to the fact that the scale function is monotonically increasing. Combining this observation with the fact that $\widetilde{V}\bigl(\scale^{-1}\bigl(\scale(z)/2\bigr)\bigr)\leq 0$, by virtue of (ii) above, we get $\widetilde{V}(x)\leq 0$.
\end{proof}

We are now in a position to prove the verification theorem that establishes the correspondence between the function $\widetilde{V}$ given by \eqref{eqsec3:ValFunc1} and the value function $\widehat{V}$ associated with the general optimal stopping problem \eqref{eqpropsec2:OptPredProb3_1}. Since $\widetilde{V}$ is the value function for the restricted problem \eqref{eqsec3:OptPredProb4}, we conclude that these two problems are equivalent. Moreover, the optimal stopping policy for Problem~\eqref{eqpropsec2:OptPredProb3_1} is given by $\tau_*=\sigma_{r_*}$, where $r_*$ is determined by \eqref{eqsec3:OptStopBnd}:

\begin{theorem}
The function $\widetilde{V}$, given by \eqref{eqsec3:ValFunc1}, corresponds with the value function $\widehat{V}$ associated with Problem~\eqref{eqpropsec2:OptPredProb3_1}. In other words, $\widetilde{V}(x)=\widehat{V}(x)$, for all $x>0$, and the optimal stopping policy for Problem~\eqref{eqpropsec2:OptPredProb3_1} is given by $\tau_*=\sigma_{r_*}$, where $r_*$ is the solution for \eqref{eqsec3:OptStopBnd}.
\end{theorem}
\begin{proof}
By inspecting \eqref{eqsec3:ValFunc1} and \eqref{eqpropsec4:FreeBoundProb2}, we see that $\widetilde{V}\notin\C{2}{0}(0,\infty)$, since
\begin{equation*}
\widetilde{V}''(r_*-)=-\frac{2}{a^2(r_*)}c(r_*)<0=\widetilde{V}''(r_*+).
\end{equation*}
Consequently, It\^o's formula cannot be applied in its conventional form to $\widetilde{V}(X)$. However, since $\widetilde{V}\in\C{2}{0}(0,r_*]\cap\C{2}{0}[r_*,\infty)$, we may apply the local time-space formula of \citet{Pes05a} to get
\begin{equation*}
\begin{split}
\widetilde{V}(X_t)&=\widetilde{V}(X_0)+\int_0^t\frac{1}{2}\Bigl(\widetilde{V}'(X_s+)+\widetilde{V}'(X_s-)\Bigr)\,\d X_s+\frac{1}{2}\int_0^t\indicator{\{X_s\neq r_*\}}\widetilde{V}''(X_s)\,\d\<X\>_s\\
&\qquad\qquad\qquad\qquad\qquad\qquad\qquad\qquad+\frac{1}{2}\int_0^t\indicator{\{X_s=r_*\}}\Bigl(\widetilde{V}'(X_s+)-\widetilde{V}'(X_s-)\Bigr)\,\d\ell^{r_*}_s(X),
\end{split}
\end{equation*}
for all $t\geq 0$, where $\ell^{r_*}(X)$ denotes the local time process of $X$ at $r_*$ \citep[see e.g.][Section~3.5]{PS06}. Next, inspection of \eqref{eqsec3:ValFunc1} and \eqref{eqpropsec4:FreeBoundProb2} reveal that $\widetilde{V}\in\C{1}{0}(0,\infty)$ and $\widetilde{V}(x)=\widetilde{V}'(x)=\widetilde{V}''(x)=0$, for all $x>r_*$. Consequently,
\begin{align*}
\widetilde{V}(X_t)&=\widetilde{V}(X_0)+\int_0^t\indicator{\{X_s<r_*\}}\widetilde{V}'(X_s)\,\d X_s+\frac{1}{2}\int_0^t\indicator{\{X_s<r_*\}}\widetilde{V}''(X_s)\,\d\<X\>_s\\
&=\widetilde{V}(X_0)+\int_0^t\indicator{\{X_s<r_*\}}\biggl(b(X_s)\widetilde{V}'(X_s)+\frac{1}{2}a^2(X_s)\widetilde{V}''(X_s)\biggr)\,\d s\\
&\qquad\qquad\qquad\qquad\qquad\qquad\qquad\qquad\qquad\qquad\qquad+\int_0^t\indicator{\{X_s<r_*\}}a(X_s)\widetilde{V}'(X_s)\,\d B_s\\
&=\widetilde{V}(X_0)-\int_0^t\indicator{\{X_s<r_*\}}c(X_s)\,\d s+\int_0^t\indicator{\{X_s<r_*\}}a(X_s)\widetilde{V}'(X_s)\,\d B_s
\end{align*}
for all $t\geq 0$, by virtue of \eqref{eqpropsec4:FreeBoundProb1a}. Rearranging this equation produces
\begin{align*}
\widetilde{V}(X_0)&=\widetilde{V}(X_t)+\int_0^tc(X_s)\,\d s-\int_0^t\indicator{\{X_s\geq r_*\}}c(X_s)\,\d s-\int_0^t\indicator{\{X_s<r_*\}}a(X_s)\widetilde{V}'(X_s)\,\d B_s\\
&\leq\int_0^tc(X_s)\,\d s-\int_0^ta(X_s)\widetilde{V}'(X_s)\,\d B_s
\end{align*}
for all $t\geq 0$, since $\widetilde{V}(X_t)\leq 0$, according to Lemma~\ref{lemsec4:ValFuncNegative}, while $c(x)\geq 0$ and $\widetilde{V}'(x)=0$, for all $x\geq r_*$. Next, fix an arbitrary integrable stopping time $\tau$, and let $(\rho_n)_{n\in\N}$ be a localising sequence of stopping times for the local martingale $\int_0^\cdot a(X_s)\widetilde{V}'(X_s)\,\d B_s$. The optional sampling theorem then yields
\begin{equation*}
\widetilde{V}(x)\leq\E_x\biggl(\int_0^{\tau\wedge\rho_n}c(X_s)\,\d s\biggr)-\E_x\biggl(\int_0^{\tau\wedge\rho_n}a(X_s)\widetilde{V}'(X_s)\,\d B_s\biggr)
=\E_x\biggl(\int_0^{\tau\wedge\rho_n}c(X_s)\,\d s\biggr),
\end{equation*}
for all $x>0$ and each $n\in\N$. Finally, the fact that $|c(x)|\leq 1$, for all $x>0$, implies that
\begin{equation*}
\E_x\biggl(\biggl|\int_0^{\tau\wedge\rho_n}c(X_s)\,\d s\biggr|\biggr)
\leq\E_x\biggl(\int_0^{\tau\wedge\rho_n}|c(X_s)|\,\d s\biggr)
\leq\E_x(\tau\wedge\rho_n)
\leq\E_x(\tau)<\infty,
\end{equation*}
for all $x>0$ and each $n\in\N$. We may thus invoke the dominated convegence theorem to get
\begin{equation*}
\widetilde{V}(x)\leq\lim_{n\uparrow\infty}\E_x\biggl(\int_0^{\tau\wedge\rho_n}c(X_s)\,\d s\biggr)
=\E_x\biggl(\int_0^\tau c(X_s)\,\d s\biggr).
\end{equation*}
The arbitrariness of $\tau$ ensures that $\widetilde{V}(x)\leq\widehat{V}(x)$, for all $x>0$, by inspection of Problem~\eqref{eqpropsec2:OptPredProb3_1}. Since the reverse inequality has already been established, the result follows.
\end{proof}
\section{Some Examples}
\label{sec5}
In this section we present a number of examples of time-homogeneous scalar diffusions $X$ that conform to Assumption~\eqref{eqsec1:Assump}. In each case we use \eqref{eqsec3:ValFunc1} and \eqref{eqsec3:OptStopBnd} to compute the value function and optimal stopping policy for Problem~\eqref{eqpropsec2:OptPredProb3_1}. We begin by considering the class of time-homogeneous scalar diffusions whose scale functions and speed measures satisfy
\begin{equation}
\label{eqsec5:ScaleSpeed2}
\scale(x)=-\alpha x^{-\mu}\qquad\text{and}\qquad\speed(\d x)=\beta x^\nu\,\d x,
\end{equation}
for all $x>0$, where $\alpha>0$, $\beta>0$, $\mu>0$ and $\nu>-1$. In that case the conditions in Assumption~\eqref{eqsec1:Assump} are easily verified, and the first-order condition \eqref{eqsec3:OptStopBnd}, which determines the optimal stopping boundary $r_*>\scale^{-1}\bigl(\scale(z)/2\bigr)>z$, assumes the following form:
\begin{subequations}
\label{eqsec5:OSB}
\begin{equation}
\label{eqsec5:OSBa}
(\nu-\mu+1)\biggl(\frac{r_*}{z}\biggr)^{\nu+1}-2(\nu+1)\biggl(\frac{r_*}{z}\biggr)^{\nu-\mu+1}+2\mu=0,
\end{equation}
if $\nu\neq\mu-1$, and
\begin{equation}
\label{eqsec5:OSBb}
\biggl(\frac{r_*}{z}\biggr)^\mu-2\mu\ln\frac{r_*}{z}-2=0,
\end{equation}
\end{subequations}
if $\nu=\mu-1$. Moreover, the value function \eqref{eqsec3:ValFunc1} may be computed explicitly, as follows:
\begin{subequations}
\label{eqsec5:ValFunc2}
\begin{equation}
\label{eqsec5:ValFunc2a}
\widetilde{V}(x)=
\begin{dcases}
\begin{split}
&\alpha\beta z^{\nu-\mu+1}\Biggl(\frac{1}{\nu-\mu+1}\biggl(\frac{r_*}{z}\biggr)^{\nu-\mu+1}-\frac{2}{\nu-2\mu+1}\biggl(\frac{r_*}{z}\biggr)^{\nu-2\mu+1}\\
&\quad+\frac{\mu}{(\nu-\mu+1)(\nu+1)}\biggl(\frac{x}{z}\biggr)^{\nu-\mu+1}\\
&\quad+\frac{2\mu}{(\nu-2\mu+1)(\nu-\mu+1)}\Biggr),
\end{split}
&\text{if $x\leq z$};\\
\begin{split}
&\alpha\beta z^{\nu-\mu+1}\Biggl(\frac{1}{\nu-\mu+1}\biggl(\frac{r_*}{z}\biggr)^{\nu-\mu+1}-\frac{2}{\nu-2\mu+1}\biggl(\frac{r_*}{z}\biggr)^{\nu-2\mu+1}\\
&\quad-\frac{\mu}{(\nu-\mu+1)(\nu+1)}\biggl(\frac{x}{z}\biggr)^{\nu-\mu+1}\\
&\quad+\frac{2\mu}{(\nu-2\mu+1)(\nu-\mu+1)}\biggl(\frac{x}{z}\biggr)^{\nu-2\mu+1}\\
&\quad+\frac{2\mu}{(\nu-\mu+1)(\nu+1)}\biggl(\frac{x}{z}\biggr)^{-\mu}\Biggr),
\end{split}
&\text{if $z\leq x\leq r_*$};\\
0,&\text{if $x\geq r_*$},
\end{dcases}
\end{equation}
for all $x>0$, in the case when $\nu\neq\mu-1$ and $\nu\neq 2\mu-1$, while
\begin{equation}
\label{eqsec5:ValFunc2b}
\widetilde{V}(x)=
\begin{dcases}
\alpha\beta\Biggl(\frac{2}{\mu}\biggl(\frac{r_*}{z}\biggr)^{-\mu}+\ln\frac{r_*}{z}+\ln\frac{x}{z}-\frac{3}{\mu}\Biggr),
&\text{if $x\leq z$};\\
\begin{split}
&\alpha\beta\Biggl(\frac{2}{\mu}\biggl(\frac{r_*}{z}\biggr)^{-\mu}-\frac{4}{\mu}\biggl(\frac{x}{z}\biggr)^{-\mu}-2\biggl(\frac{x}{z}\biggr)^{-\mu}\ln\frac{x}{z}+\ln\frac{r_*}{z}\\
&\quad-\ln\frac{x}{z}+\frac{1}{\mu}\Biggr),
\end{split}
&\text{if $z\leq x\leq r_*$};\\
0,&\text{if $x\geq r_*$},
\end{dcases}
\end{equation}
for all $x>0$, in the case when $\nu=\mu-1$, and
\begin{equation}
\label{eqsec5:ValFunc2c}
\widetilde{V}(x)=
\begin{dcases}
\alpha\beta z^\mu\Biggl(\frac{1}{\mu}\biggl(\frac{r_*}{z}\biggr)^\mu+\frac{1}{2\mu}\biggl(\frac{x}{z}\biggr)^\mu-2\ln\frac{r_*}{z}-\frac{2}{\mu}\Biggr),
&\text{if $x\leq z$};\\
\begin{split}
&\alpha\beta z^\mu\Biggl(\frac{1}{\mu}\biggl(\frac{r_*}{z}\biggr)^\mu-\frac{1}{2\mu}\biggl(\frac{x}{z}\biggr)^\mu+\frac{1}{\mu}\biggl(\frac{x}{z}\biggr)^{-\mu}-2\ln\frac{r_*}{z}\\
&\quad+2\ln\frac{x}{z}-\frac{2}{\mu}\Biggr),
\end{split}
&\text{if $z\leq x\leq r_*$};\\
0,&\text{if $x\geq r_*$},
\end{dcases}
\end{equation}
for all $x>0$, in the case when $\nu=2\mu-1$.

We now consider a number of specific examples of time-homogeneous scalar diffusions whose scale functions and speed measures conform to the representation \eqref{eqsec5:ScaleSpeed2}. We begin with the family of transient Bessel processes:
\end{subequations}

\begin{example}
\label{exsec5:Bes}
Let $X$ be a Bessel process of dimension $\delta>2$, in which case
\begin{equation*}
b(x)=\frac{\delta-2}{2x}\qquad\text{and}\qquad a(x)=1, 
\end{equation*}
for all $x>0$. According to \eqref{eqsec1:ScaleSpeed1}, its scale function and speed measure are given by
\begin{equation*}
\scale(x)=-x^{-(\delta-2)}\qquad\text{and}\qquad\speed(\d x)=\frac{2}{\delta-2}x^{\delta-1}\,\d x,
\end{equation*}
for all $x>0$ \citep[see e.g.][Section~XI.1]{RY99}. From \eqref{eqsec5:OSBa} it follows that the equation for the optimal stopping boundary $r_*$ is
\begin{equation}
\label{eqsec5:BesOSB}
\biggl(\frac{r_*}{z}\biggr)^\delta-\delta\biggl(\frac{r_*}{z}\biggr)^2+\delta=2.
\end{equation}
Using \eqref{eqsec5:ValFunc2a}, we see that the value function is given by
\begin{subequations}
\label{eqsec5:BesValFunc}
\begin{equation}
\label{eqsec5:BesValFunca}
\widetilde{V}(x)=
\begin{dcases}
\begin{split}
&z^2\Biggl(\frac{1}{\delta-2}\biggl(\frac{r_*}{z}\biggr)^2-\frac{4}{(\delta-2)(4-\delta)}\biggl(\frac{r_*}{z}\biggr)^{4-\delta}+\frac{1}{\delta}\biggl(\frac{x}{z}\biggr)^2\\
&\quad+\frac{2}{4-\delta}\Biggr),
\end{split}
&\text{if $x\leq z$};\\
\begin{split}
&z^2\Biggl(\frac{1}{\delta-2}\biggl(\frac{r_*}{z}\biggr)^2-\frac{4}{(\delta-2)(4-\delta)}\biggl(\frac{r_*}{z}\biggr)^{4-\delta}-\frac{1}{\delta}\biggl(\frac{x}{z}\biggr)^2\\
&\quad+\frac{2}{4-\delta}\biggl(\frac{x}{z}\biggr)^{4-\delta}+\frac{2}{\delta}\biggl(\frac{x}{z}\biggr)^{2-\delta}\Biggr),
\end{split}
&\text{if $z\leq x\leq r_*$};\\
0,&\text{if $x\geq r_*$},
\end{dcases}
\end{equation}
for all $x>0$, when $\delta\neq 4$, while \eqref{eqsec5:ValFunc2c} yields
\begin{equation}
\label{eqsec5:BesValFuncb}
\widetilde{V}(x)=
\begin{dcases}
z^2\Biggl(\frac{1}{2}\biggl(\frac{r_*}{z}\biggr)^2+\frac{1}{4}\biggl(\frac{x}{z}\biggr)^2-1-2\ln\frac{r_*}{z}\Biggr),
&\text{if $x\leq z$};\\
z^2\Biggl(\frac{1}{2}\biggl(\frac{r_*}{z}\biggr)^2-\frac{1}{4}\biggl(\frac{x}{z}\biggr)^2-1+\frac{1}{2}\biggl(\frac{x}{z}\biggr)^{-2}+2\ln\frac{x}{z}-2\ln\frac{r_*}{z}\Biggr),
&\text{if $z\leq x\leq r_*$};\\
0,&\text{if $x\geq r_*$},
\end{dcases}
\end{equation}
\end{subequations}
for all $x>0$, when $\delta=4$. For the concrete example of a Bessel process of dimension three, \eqref{eqsec5:BesOSB} possesses three solutions, only one of which satisfies the admissibility criterion $r_*>z$, namely
\begin{equation*}
r_*=\biggl(1+2\cos\frac{\pi}{9}\biggr)z\approx 2.87939z.
\end{equation*}
Figure~\ref{figsec5:Bes3ValFunc} plots the associated value function \eqref{eqsec5:BesValFunca}, for the case when $z=1$.
\end{example}

\begin{figure}[h]
\centering
\includegraphics[scale=1]{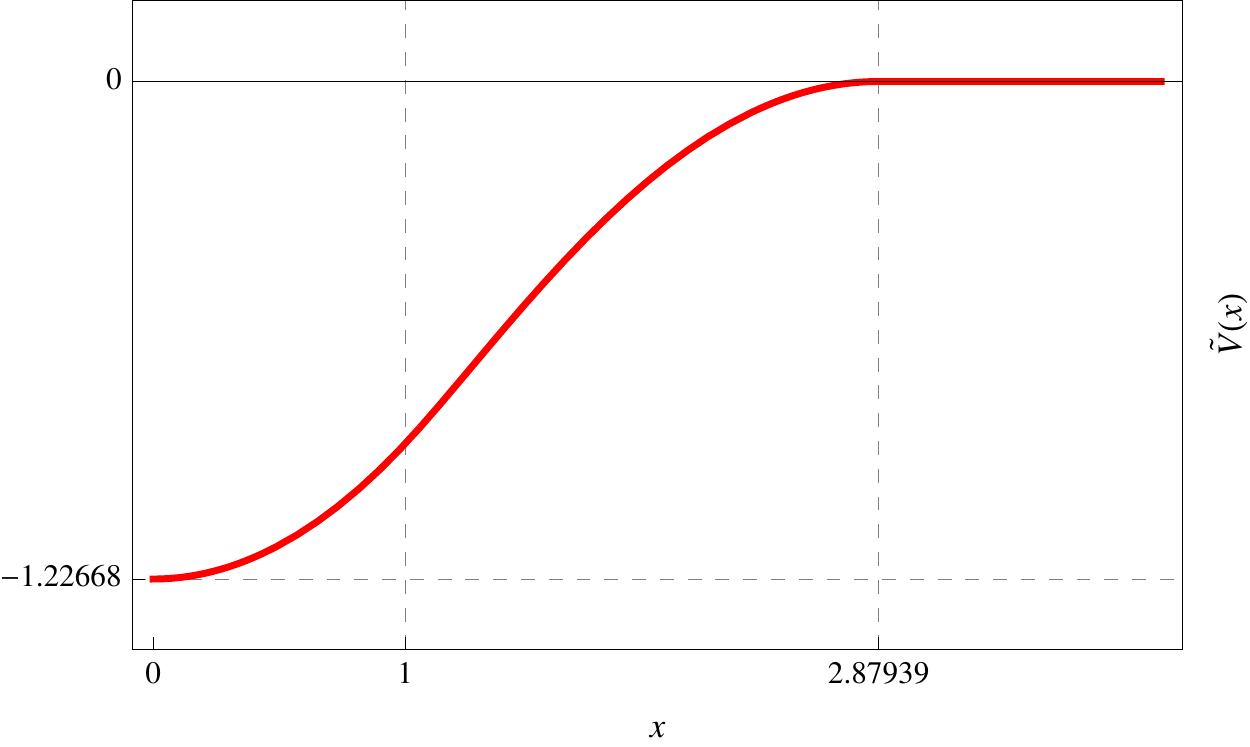}
\caption{The value function \eqref{eqsec5:BesValFunca}, for a transient Bessel process of dimension $\delta=3$, with target level $z=1$. The optimal stopping boundary is $r_*\approx 2.87939$.}
\label{figsec5:Bes3ValFunc}
\end{figure}

A notable feature of Example~\ref{exsec5:Bes} is the linearity of the optimal stopping boundary $r_*$ for transient Bessel processes with respect to the target level $z$, as evident from \eqref{eqsec5:BesOSB}. In fact, inspection of \eqref{eqsec5:OSBa}--\eqref{eqsec5:OSBb} reveals that linearity of $r_*$ with respect to to $z$ is a general feature of diffusions conforming to \eqref{eqsec5:ScaleSpeed2}. In addition, \eqref{eqsec5:BesOSB} also reveals that $\partial r_*/\partial\delta<0$, from which it follows that the optimal stopping boundaries are lower for transient Bessel processes with higher dimensions. The underlying intuition is clear, since Bessel processes with higher dimensions exhibit greater (positive) drifts, which in turn decreases their last-passage times to the target level. Consequently, the first-passage times that best approximate the last-passage times to the target level correspond to lower levels, for transient Bessel processes of higher dimensions.

Transient squared Bessel processes represent another prominent class of time-homogeneous diffusions whose scale functions and speed measures conform to \eqref{eqsec5:ScaleSpeed2}:

\begin{example}
\label{exsec5:SqBes}
Let $X$ be a squared Bessel process of dimension $\delta>2$, in which case
\begin{equation*}
b(x)=\delta\qquad\text{and}\qquad a(x)=2\sqrt{x},
\end{equation*}
for all $x>0$. According to \eqref{eqsec1:ScaleSpeed1}, its scale function and speed measure are given by
\begin{equation*}
\scale(x)=-x^{-\frac{\delta-2}{2}}\qquad\text{and}\qquad\speed(\d x)=\frac{1}{\delta-2}x^{\frac{\delta-2}{2}}\,\d x,
\end{equation*}
for all $x>0$ \citep[see e.g.][Section~XI.1]{RY99}. From \eqref{eqsec5:OSBa} it follows that the equation for the optimal stopping boundary $r_*$ is
\begin{equation}
\label{eqsec5:SqBesOSB}
\biggl(\frac{r_*}{z}\biggr)^\frac{\delta}{2}-\delta\frac{r_*}{z}+\delta=2.
\end{equation}
Using \eqref{eqsec5:ValFunc2a}, we see that the value function is given by
\begin{subequations}
\label{eqsec5:SqBesValFunc}
\begin{equation}
\label{eqsec5:SqBesValFunca}
\widetilde{V}(x)=
\begin{dcases}
\begin{split}
&z\Biggl(\frac{1}{\delta-2}\frac{r_*}{z}-\frac{4}{(\delta-2)(4-\delta)}\biggl(\frac{r_*}{z}\biggr)^\frac{4-\delta}{2}+\frac{1}{\delta}\frac{x}{z}+\frac{2}{4-\delta}\Biggr),
\end{split}
&\text{if $x\leq z$};\\
\begin{split}
&z\Biggl(\frac{1}{\delta-2}\frac{r_*}{z}-\frac{4}{(\delta-2)(4-\delta)}\biggl(\frac{r_*}{z}\biggr)^\frac{4-\delta}{2}-\frac{1}{\delta}\frac{x}{z}+\frac{2}{4-\delta}\biggl(\frac{x}{z}\biggr)^\frac{4-\delta}{2}\\
&\quad+\frac{2}{\delta}\biggl(\frac{x}{z}\biggr)^\frac{2-\delta}{2}\Biggr),
\end{split}
&\text{if $z\leq x\leq r_*$};\\
0,&\text{if $x\geq r_*$},
\end{dcases}
\end{equation}
for all $x>0$, when $\delta\neq 4$, while \eqref{eqsec5:ValFunc2c} yields
\begin{equation}
\label{eqsec5:SqBesValFuncb}
\widetilde{V}(x)=
\begin{dcases}
z\Biggl(\frac{1}{2}\frac{r_*}{z}+\frac{1}{4}\frac{x}{z}-1-\ln\frac{r_*}{z}\Biggr),&\text{if $x\leq z$};\\
z\Biggl(\frac{1}{2}\frac{r_*}{z}-\frac{1}{4}\frac{x}{z}-1+\frac{1}{2}\biggl(\frac{x}{z}\biggr)^{-1}+\ln\frac{x}{z}-\ln\frac{r_*}{z}\Biggr),&\text{if $z\leq x\leq r_*$};\\
0,&\text{if $x\geq r_*$},
\end{dcases}
\end{equation}
\end{subequations}
for all $x>0$, when $\delta=4$. For the concrete example of a squared Bessel process of dimension four, \eqref{eqsec5:SqBesOSB} possesses two solutions, one of which satisfies the admissibility criterion $r_*>z$, namely
\begin{equation*}
r_*=\bigl(2+\sqrt{2}\bigr)z\approx 3.41421 z.
\end{equation*}
Figure~\ref{figsec5:SqBes4ValFunc} plots the associated value function \eqref{eqsec5:SqBesValFuncb}, for the case when $z=1$.
\end{example}

\begin{figure}[h]
\centering
\includegraphics[scale=1]{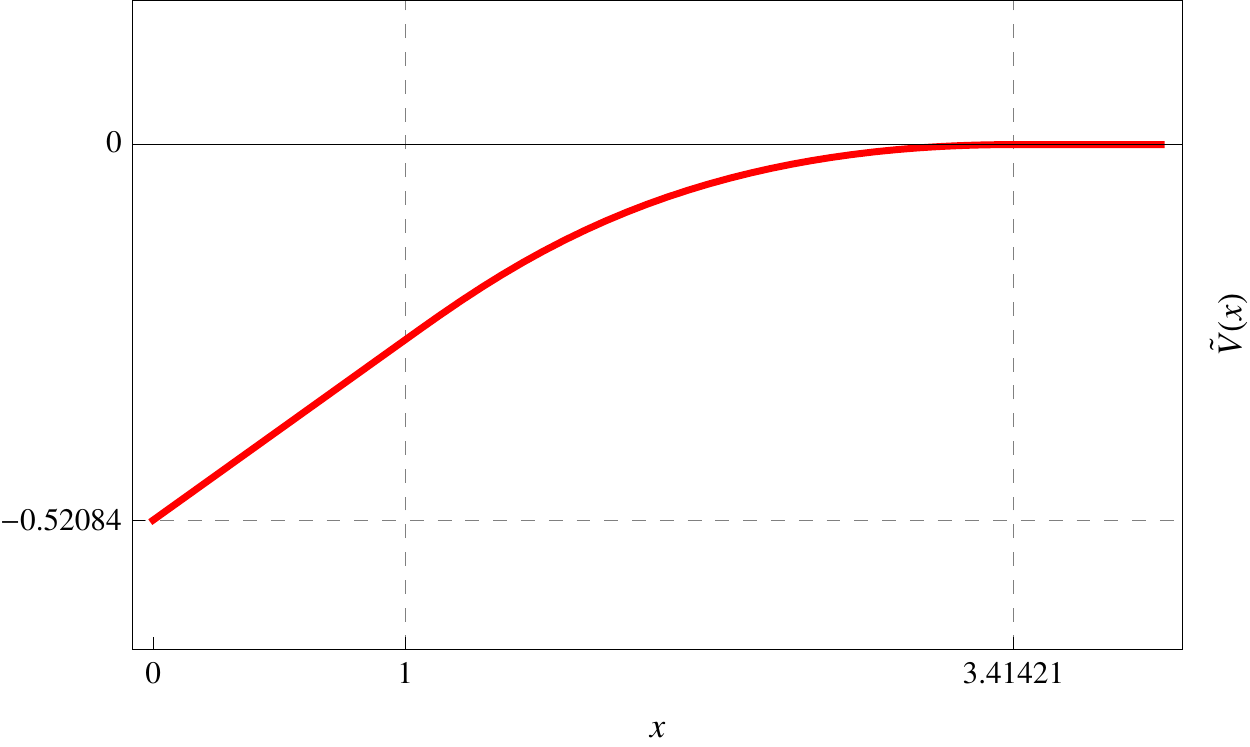}
\caption{The value function \eqref{eqsec5:SqBesValFuncb}, for a transient squared Bessel process of dimension $\delta=4$, with target level $z=1$. The optimal stopping boundary is $r_*\approx 3.41421$.}
\label{figsec5:SqBes4ValFunc}
\end{figure}

Note that the origin is an entrance boundary for the processes considered in Examples~\ref{exsec5:Bes} and \ref{exsec5:SqBes} \citep[see e.g.][Appendix~1.21 and 1.23]{BS02}, which explains why $\widetilde{V}(0+)>-\infty$ in both of those cases (see Proposition~\ref{propsec4:FreeBoundProb}). By contrast, a transient geometric Brownian motion provides an example of a time-homogeneous diffusion whose scale function and speed measure admit the representation \eqref{eqsec5:ScaleSpeed2}, and for which the origin is a non-attractive natural boundary \citep[see e.g.][Appendix~1.20]{BS02}:

\begin{example}
\label{exsec5:GBM}
Let $X$ be a geometric Brownian motion, in which case
\begin{equation*}
b(x)=\lambda x\qquad\text{and}\qquad a(x)=\sigma x,
\end{equation*}
for all $x>0$, where $\sigma>0$ and $\lambda\in\R$. We impose the condition
\begin{equation*}
\kappa\coloneqq\frac{\lambda}{\sigma^2}-\frac{1}{2}>0,
\end{equation*}
which ensures that the process is transient \citep[see e.g.][Appendix~1.20]{BS02}. According to \eqref{eqsec1:ScaleSpeed1}, its scale function and speed measure are given by
\begin{equation*}
\scale(x)=-\frac{1}{2\kappa}x^{-2\kappa}\qquad\text{and}\qquad\speed(\d x)=\frac{1}{\sigma^2}x^{2\kappa-1}\,\d x,
\end{equation*}
for all $x>0$ \citep[see e.g.][Appendix~1.20]{BS02}. From \eqref{eqsec5:OSBb} it follows that the equation for the optimal stopping boundary is
\begin{equation}
\label{eqsec5:GBMOSB}
\biggl(\frac{r_*}{z}\biggr)^{2\kappa}-4\kappa\ln\frac{r_*}{z}-2=0.
\end{equation}
Using \eqref{eqsec5:ValFunc2b}, we see that the value function is given by
\begin{equation}
\label{eqsec5:GBMValFunc}
\widetilde{V}(x)=
\begin{dcases}
\frac{1}{2\kappa\sigma^2}\Biggl(\frac{1}{\kappa}\biggl(\frac{r_*}{z}\biggr)^{-2\kappa}+\ln\frac{r_*}{z}+\ln\frac{x}{z}-\frac{3}{2\kappa}\Biggr),
&\text{if $x\leq z$};\\
\begin{split}
&\frac{1}{2\kappa\sigma^2}\Biggl(\frac{1}{\kappa}\biggl(\frac{r_*}{z}\biggr)^{-2\kappa}-\frac{2}{\kappa}\biggl(\frac{x}{z}\biggr)^{-2\kappa}-2\biggl(\frac{x}{z}\biggr)^{-2\kappa}\ln\frac{x}{z}+\ln\frac{r_*}{z}\\
&\quad-\ln\frac{x}{z}+\frac{1}{2\kappa}\Biggr),
\end{split}
&\text{if $z\leq x\leq r_*$};\\
0,&\text{if $x\geq r_*$},
\end{dcases}
\end{equation}
for all $x>0$. For the concrete example of a geometric Brownian motion with $\lambda=1$ and $\sigma=1$, the transcendental equation \eqref{eqsec5:GBMOSB} possesses two solutions, one of which satisfies the admissibility condition $r_*>z$, namely $r_*\approx 5.35370 z$. Figure~\ref{figsec5:GBMValFunc} plots the associated value function \eqref{eqsec5:GBMValFunc}, for the case when $z=1$.
\end{example}

\begin{figure}[h]
\centering
\includegraphics[scale=1]{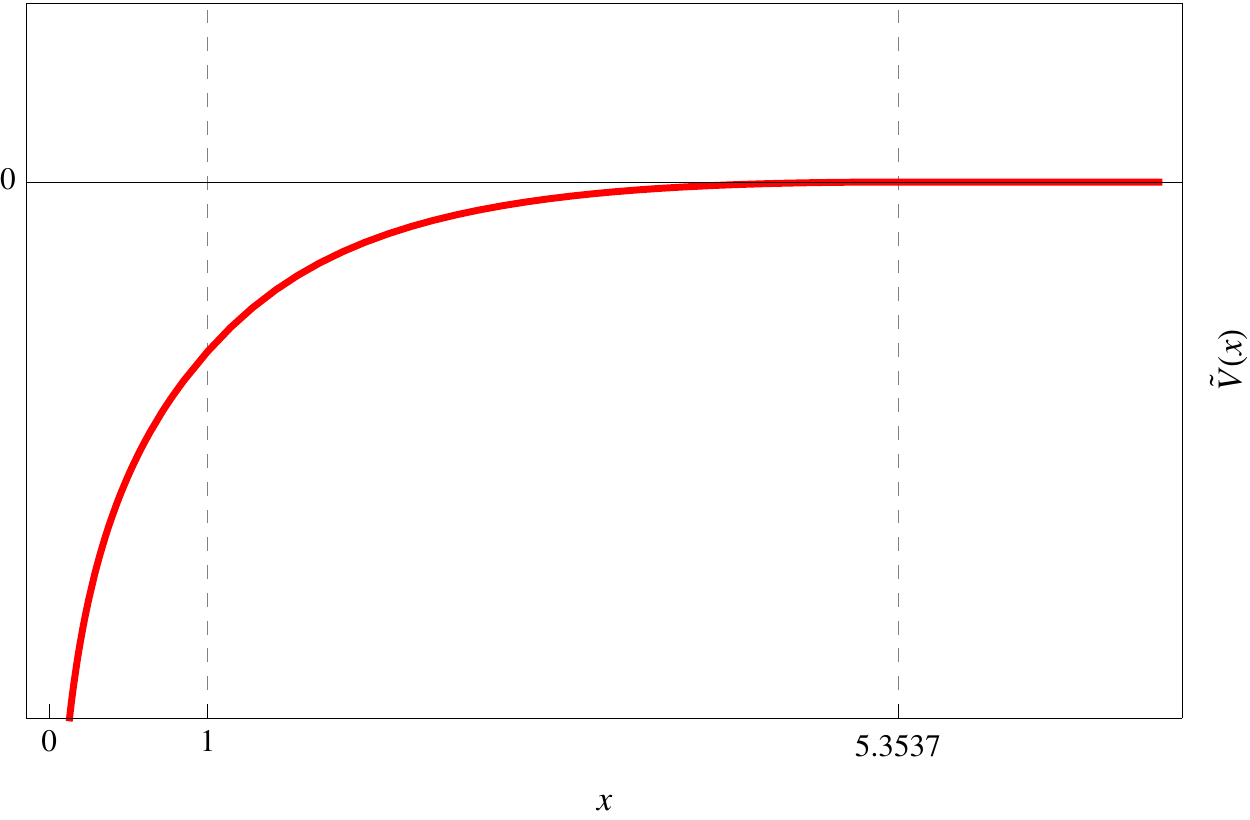}
\caption{The value function \eqref{eqsec5:GBMValFunc}, for a transient geometric Brownian motion with $\lambda=1$, $\sigma=1$, and target level $z=1$. The optimal stopping boundary is $r_*\approx 5.35370$.}
\label{figsec5:GBMValFunc}
\end{figure}

Inspection of \eqref{eqsec5:GBMOSB} reveals that $\partial r_*/\partial\kappa<0$, for the transient geometric Brownian motion considered in Example~\ref{exsec5:GBM}. In other words, a higher drift rate $\lambda$ and/or a lower volatility $\sigma$ correspond to a lower optimal stopping boundary. The intuition is obvious, once again, since a higher drift rate and/or a lower volatility decrease the last-passage time of the process to the target level. Also note that $\widetilde{V}(0+)=-\infty$, since the origin is a non-attractive natural boundary for transient geometric Brownian motions (see Proposition~\ref{propsec4:FreeBoundProb}).

The processes in Examples~\ref{exsec5:Bes}--\ref{exsec5:GBM} are all non-explosive, in the sense that $\P_x(\zeta<\infty)=0$, for all $x>0$. Next, we present an example of a transient diffusion with explosive sample paths, motivated by a process studied by \citet{KR13}. In this example, however, the scale function and speed measure are not of the form \eqref{eqsec5:ScaleSpeed2}:

\begin{example}
\label{exsec5:ExplProc}
Let $X$ be determined by
\begin{equation*}
b(x)=\lambda\kappa x^p+\frac{1}{2}\kappa^2px^{2p-1}\qquad\text{and}\qquad a(x)=\kappa x^p,
\end{equation*}
for all $x>0$, where $\lambda>0$, $\kappa>0$ and $p>1$. According to \eqref{eqsec1:ScaleSpeed1} its scale function and speed measure are given by
\begin{equation*}
\scale(x)=1-\e^{-\frac{2\lambda}{\kappa(1-p)}x^{1-p}}
\qquad\text{and}\qquad
\speed(\d x)=\frac{1}{\lambda\kappa}x^{-p}\e^{\frac{2\lambda}{\kappa(1-p)}x^{1-p}}\,\d x,
\end{equation*}
for all $x>0$. While the scale function and speed measure specified above do not conform to \eqref{eqsec5:ScaleSpeed2}, the conditions in Assumption~\eqref{eqsec1:Assump} are nevertheless easily verified. Moreover, we observe that
\begin{equation*}
\int_1^\infty\bigl(\scale(\infty-)-\scale(y)\bigr)\,\speed(\d y)
=\frac{1}{\lambda\kappa}\int_1^\infty y^{-p}\biggl(1-\e^{\frac{2\lambda}{\kappa(1-p)}y^{1-p}}\biggr)\,\d y
<-\frac{1}{\lambda\kappa}\frac{1}{1-p}<\infty,
\end{equation*}
which indicates that the boundary at infinity is attainable \citep[see e.g.][Definition~6.2]{KT81}. In other words, $\P_x(\zeta<\infty)>0$, for all $x>0$. Although the optimal stopping boundary $r_*$ and the value function $\widetilde{V}$ cannot be determined analytically in this example, we are nevertheless able to evaluate them numerically from \eqref{eqsec3:ValFunc1} and \eqref{eqsec3:OptStopBnd}. In the case when $z=1$, numerical solution of \eqref{eqsec3:OptStopBnd} gives $r_*\approx 3.77882$, for the parameter $\lambda=1$, $\kappa=1$, and $p=2$. Figure~\ref{figsec5:ExplValFunc} employs numerical integration to plot the associated value function \eqref{eqsec3:ValFunc1}.
\end{example}

\begin{figure}[h]
\centering
\includegraphics[scale=1]{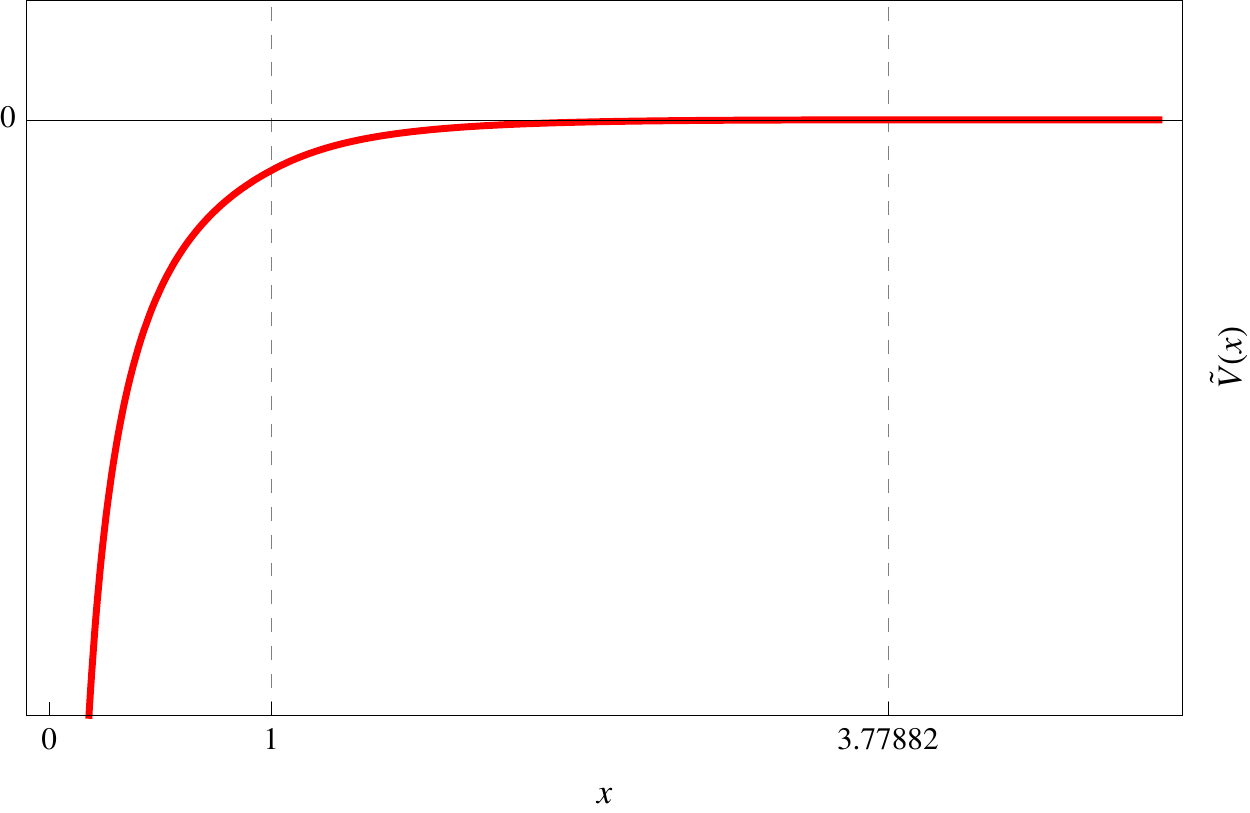}
\caption{The value function for the explosive process in Example~\ref{exsec5:ExplProc} with $\lambda=1$, $\kappa=1$, $p=2$, and target level $z=1$. The optimal stopping boundary is $r_*\approx 3.7782$.}
\label{figsec5:ExplValFunc}
\end{figure}

\bibliography{ProbFinBiblio}
\bibliographystyle{chicago}
\end{document}